\documentclass[12pt]{article}

\usepackage[usenames,dvipsnames]{xcolor}
\usepackage{amssymb}
\usepackage{amsmath}
\usepackage{graphicx}

\usepackage[colorlinks=true,
linkcolor=webgreen,
filecolor=webbrown,
citecolor=webgreen]{hyperref}

\definecolor{webgreen}{rgb}{0,.5,0}
\definecolor{webbrown}{rgb}{.6,0,0}

\usepackage{fullpage}
\usepackage{float}
\usepackage[font=small]{caption}

\usepackage{epstopdf}
\usepackage{amsthm}
\usepackage{stmaryrd}
\usepackage{enumerate}

\usepackage{tikz}
\usepackage{xparse}
\usepackage{scalefnt}
\usetikzlibrary{matrix,backgrounds}
\pgfdeclarelayer{myback}
\pgfsetlayers{myback,background,main}
\tikzset{myfillcolor/.style = {fill=#1}}
\NewDocumentCommand{\fhighlight}{O{blue!40 } m m}{\draw[myfillcolor=#1] (#2.north west)rectangle (#3.south east);}

\newcommand{\seqnum}[1]{\href{http://oeis.org/#1}{\underline{#1}}}

\newcommand{\Na}{\mathbb N^\ast}
\newcommand{\No}{\mathbb N_0}

\begin{document}

\theoremstyle{plain}
\newtheorem{theorem}{Theorem}
\newtheorem{lemma}{Lemma}
\newtheorem{corollary}{Corollary}
\theoremstyle{definition}
\newtheorem{definition}{Definition}

\theoremstyle{remark}
\newtheorem{remark}{Remark}

\begin{center}
\vskip 1cm{\LARGE\bf
Links Between Sums Over Paths in Bernoulli's Triangles\\
\vskip .075in
and the Fibonacci Numbers
}
\vskip 1cm
{\large
Denis Neiter and Amsha Proag\\
Ecole Polytechnique\\
Route de Saclay\\
91128 Palaiseau\\
France\\
\href{mailto:denis.neiter@polytechnique.org}{\tt denis.neiter@polytechnique.org}\\
\href{mailto:amsha.proag@polytechnique.org}{\tt amsha.proag@polytechnique.org}
}
\end{center}

\vskip .2 in

\begin{abstract}
We investigate paths in Bernoulli's triangles and derive several relations linking the partial sums of binomial coefficients to the Fibonacci numbers.
\end{abstract}

\section{Introduction}

Binomial coefficients appear in many identities, some of which are closely connected to the Fibonacci sequence~\cite{azarian2012a,azarian2012b}. Pascal's triangle has been explored for links to the Fibonacci sequence as well as to generalized sequences~\cite{hoggattjr1970}. The partial sums of the binomial coefficients are less well known, although a number of identities have been found regarding sums of their powers~\cite{calkin1994,hirschhorn1996} and polynomials~\cite{he2014}. To add to the existing corpus, we review Bernoulli's second and third-order triangles for relations pertaining to sums of the binomial coefficients. We contribute several relations that link the Fibonacci numbers to binomial partial sums.

\section{Notation, definitions and preliminary lemma}

We let $\Na$ denote the positive natural numbers, i.e., $\{1,2,\ldots\}$ and we let $\No$ refer to $\Na\cup\{0\}$. For $p,q\in\mathbb Z,\;p\leqslant q$, we let $\llbracket p,q \rrbracket$ denote the integers between $p$ and $q$, i.e., $\llbracket p,q \rrbracket=\{k\in\mathbb Z\mid p\leqslant k\leqslant q\}$. For $p\in\mathbb Z$ we also let $p^+$ denote $\frac12\left(p+\vert p\vert\right)$.

We let $(F_n)_{n \in \No}$ denote the \textit{Fibonacci numbers} (\seqnum{A000045}~\cite{oeis}).

\begin{definition}Let $B_{n,k}$ denote the sum of the first $k$ binomial coefficients, i.e., $$\forall n,k\in\No,\;B_{n,k} = \displaystyle\sum_{q=0}^{k}\binom nq.$$
\label{def_Bernoulli}
\end{definition}
Definition~\ref{def_Bernoulli} implies that $\forall n\in\No,\forall k\geqslant 1,\;B_{n,k}=B_{n,k-1}+\displaystyle\binom nk$ and $\forall k\geqslant n$, $B_{n,k}=2^n$. We can therefore derive the following recurrence relation:
\begin{eqnarray}
  \forall n,k\geqslant1,\;B_{n,k}&=&\displaystyle\sum_{q=0}^k\binom nq\nonumber\\
  &=&\displaystyle\sum_{q=0}^k\binom {n-1}q+\displaystyle\sum_{q=0}^k\binom {n-1}{q-1}\nonumber\\
  &=&\displaystyle\sum_{q=0}^k\binom {n-1}q+\displaystyle\sum_{q=0}^{k-1}\binom {n-1}{q},\nonumber\\
B_{n,k}&=&B_{n-1,k}+B_{n-1,k-1}.
\label{relrecBnk}
\end{eqnarray}
We recognize the structure of Pascal's rule for the binomial coefficients with the boundary values $B_{n,0}=1$ and $B_{n,n}=\displaystyle\sum_{q=0}^n\binom nq=2^n$.

\begin{definition}
Let \textit{Bernoulli's triangle} denote the triangle formed by $\left(B_{n,k}\right)_{n\in\No,k\in\llbracket0,n\rrbracket}$. Bernoulli's triangle is illustrated by Figure~\ref{figtriangleB2}.
\end{definition}

\begin{figure}[!b]
\begin{center}
{\scalefont{0.9}
\begin{tikzpicture}
\matrix[matrix of math nodes,row sep=2pt,column sep = 2pt] (m) {
&k&0&1&2&3&4&5&6&7&8&9\\
n\\
0&&1\\
1&&1&2\\
2&&1&3&4\\
3&&1&4&7&8\\
4&&1&5&11&15&16\\
5&&1&6&16&26&31&32\\
6&&1&7&22&42&57&63&64\\
7&&1&8&29&64&99&120&127&128\\
8&&1&9&37&93&163&219&247&255&256\\
9&&1&10&46&130&256&382&466&502&511&512\\
};
\end{tikzpicture}
}
\captionof{figure}{Bernoulli's triangle (\seqnum{A008949}~\cite{oeis}): the element in row $n$ and column $k$ corresponds to $B_{n,k}$.}
\label{figtriangleB2}
\end{center}
\end{figure}

\begin{definition} For $n,k\in\No$, let $B^{[1]}_{n,k}=\displaystyle\binom nk$. Then for every $m \geqslant 1$ we define $$\forall n,k\in\No,\;B^{[m+1]}_{n,k} = \displaystyle\sum_{q=0}^{k}B^{[m]}_{n,q}.$$
For $m\geqslant 1$, let \textit{Bernoulli's $m$th-order triangle} $\mathcal B^{[m]}$ denote the triangle formed by the family $\left(B^{[m]}_{n,k}\right)_{n\in\No, k\in\llbracket0,n\rrbracket}$.
\label{defBm}
\end{definition}
The first and second-order triangles $\mathcal B^{[1]}$ and $\mathcal B^{[2]}$ correspond respectively to Pascal's triangle and Bernoulli's triangle. For $m\geqslant1$, the elements of $\mathcal B^{[m]}$ verify the following recurrence relation:
\begin{eqnarray}
\forall n\geqslant2,\forall k\in\llbracket1,n-1\rrbracket,\;B^{[m]}_{n,k}=B^{[m]}_{n-1,k}+B^{[m]}_{n-1,k-1}.
\label{relrecBm}
\end{eqnarray}
The proof of Relation~(\ref{relrecBm}) is the same as that of Relation~(\ref{relrecBnk}). The corresponding boundary values are $\forall n \in\No, \forall m\geqslant1,\;B^{[m]}_{n,0}=1$ and $\forall m\geqslant2,\;B^{[m]}_{n,n}=\displaystyle\sum_{q=0}^nB^{[m-1]}_{n,q}$.

\begin{definition}
In Bernoulli's $m$th-order triangle $\mathcal B^{[m]}$, for $c,l\in\mathbb Z,\;n_0\in\No,\;k_0\in\llbracket0,n_0\rrbracket$, let \textit{the path following direction $(c,l)$ from $(n_0,k_0)$} denote the sequence $\left(B^{[m]}_{n_0+kl,k_0-kc}\right)_{k\in\No}$ (where $l>0$ corresponds to increasing row numbers from $n_0$, and $c<0$ to increasing column numbers from $k_0$). For a given pair $(n_0,k_0)$, $\left(B^{[m]}_{n_0+kl,k_0-kc}\right)_{k\in\No}$ contains a finite number of elements when either $l\leqslant0$, or $l>0$ and $c\notin\llbracket-l,0\rrbracket$.
\end{definition}

In this work, we focus our attention on sums of elements along the two types of paths defined below.
\begin{definition}Let $m\in\Na$, $n\in\No$.
\begin{enumerate}[(i)]
\item Let $S^{[m]}_n\left(c,l\right)$ denote the sum over the path following direction $(c,l)$ ($c>0$, $l<0$) from $(n,n)$, i.e.,
$$S^{[m]}_n\left(c,l\right)=\displaystyle\sum_{k=0}^{\left\lfloor n/c\right\rfloor}B^{[m]}_{n+kl,n-kc}.$$
We shall also use $$\bar S^{[m]}_n\left(c,l\right)=2B^{[m]}_{n,n}-S^{[m]}_n\left(c,l\right)=B^{[m]}_{n,n}-\displaystyle\sum_{k=1}^{\left\lfloor n/c\right\rfloor}B^{[m]}_{n+kl,n-kc}.$$
\item Let $T^{[m]}_n\left(c,l\right)$ denote the sum over the path following direction $(c,l)$ ($c<0$, $l<0$) from $(n,0)$, i.e.,
$$T^{[m]}_n\left(c,l\right)=\displaystyle\sum_{k=0}^{\left\lfloor -n/(c+l)\right\rfloor}B^{[m]}_{n+kl,-kc}.$$
\end{enumerate}
\label{defpathsST}
\end{definition}

\begin{lemma}
Let $u\in\mathbb Z^{\No}$ and $\forall n\in\Na,\;v_n=u_n-2u_{n-1}$.

Then $\forall n\in\No,\;u_n=2^nu_0+\displaystyle\sum_{k=1}^n2^{n-k}v_k$.
\label{lemma}
\end{lemma}
\begin{proof}
From the definition of $(v_n)_{n\in\Na}$, we have
\begin{eqnarray*}
\displaystyle\sum_{k=1}^n2^{n-k}v_k&=&\sum_{k=1}^n2^{n-k}(u_k-2u_{k-1})\\
&=&\sum_{k=1}^n2^{n-k}u_k-\sum_{k=1}^n2^{n-k+1}u_{k-1}\\
&=&\sum_{k=1}^n2^{n-k}u_k-\sum_{k=0}^{n-1}2^{n-k}u_k\\
&=&u_n-2^nu_0.
\end{eqnarray*}
\end{proof}

Applying Lemma~\ref{lemma} to $u_n=F_{n+2}$ (i.e., $v_n=-F_{n-1}$) or $u_n=F_{2n+4}$ (i.e., $v_n=F_{2n+1}$) yields the following relations for $n\in\No$:\begin{eqnarray}
F_{n+2}=2^n-\sum_{k=1}^n2^{n-k}F_{k-1},\label{relsumfibo}\\
F_{2n+4}=3\cdot2^n+\sum_{k=1}^n2^{n-k}F_{2k+1}.\label{relsumfibo2}
\end{eqnarray}
Note that Benjamin and Quinn have reported Relation~(\ref{relsumfibo})~\cite[Id.\ 10]{benjaminquinn2003}.

\section{Bernoulli's second-order triangle}

\subsection{The path following direction \texorpdfstring{$(2,-1)$}{(2,-1)} from \texorpdfstring{$(n,n)$}{(n,n)}}

We first consider the path following direction $(2,-1)$ from $(n,n)$ in Bernoulli's triangle $\mathcal B^{[2]}$. As illustrated by Figure~\ref{figdiagonales}, the sequence $\left(\bar S^{[2]}_n\left( 2,-1\right)\right)_{n\in\No}$ can be related to the Fibonacci sequence by $\bar S^{[2]}_n\left( 2,-1\right)=F_{n+2}$. Given that $\bar S^{[2]}_n\left( 2,-1\right)=2B_{n,n}-S^{[2]}_n\left( 2,-1\right)$, we find that $S^{[2]}_n\left( 2,-1\right)=2^{n+1}-F_{n+2}$, which we formally express as Theorem~\ref{theoremS2} below.

\begin{figure}[!t]
\begin{center}
{\scalefont{0.9}
\begin{tikzpicture}[baseline=-\the\dimexpr\fontdimen22\textfont2\relax ]

\matrix[matrix of math nodes,row sep=5pt,column sep = 5pt,
 column 12/.style={anchor=base west},] (m) {
1&&&&&&&&&&\quad&1=1\\
1&2&&&&&&&&&\quad&2=2\\
1&3&4&&&&&&&&\quad&3=4-1\\
1&4&7&8&&&&&&&\quad&5=8-3\\
1&5&11&15&16&&&&&&\quad&8=16-7-1\\
1&6&16&26&31&32&&&&&\quad&13=32-15-4\\
1&7&22&42&57&63&64&&&&\quad&21=64-31-11-1\\
1&8&29&64&99&120&127&128&&&\quad&34=128-63-26-5\\
1&9&37&93&163&219&247&255&256&&\quad&55=256-127-57-16-1\\
1&10&46&130&256&382&466&502&511&512&\quad&89=512-255-120-42-6\\
};

\begin{pgfonlayer}{myback}
\foreach \element in {m-3-2,m-4-4,m-4-12}{
\fhighlight[blue!20]{\element}{\element}
}
\foreach \element in {m-3-1,m-4-3,m-5-5,m-5-12}{
\fhighlight[red!20]{\element}{\element}
}
\foreach \element in {m-4-2,m-5-4,m-6-6,m-6-12}{
\fhighlight[green!20]{\element}{\element}
}
\foreach \element in {m-4-1,m-5-3,m-6-5,m-7-7,m-7-12}{
\fhighlight[orange!20]{\element}{\element}
}
\end{pgfonlayer}

\end{tikzpicture}
}
\captionof{figure}{Four instances of the path following direction $(2,-1)$ from $(n,n)$ in Bernoulli's triangle are highlighted in distinct colors. Each term of $\bar S^{[2]}_n\left(2,-1\right)$ corresponds to a Fibonacci number.}
\label{figdiagonales}
\end{center}
\end{figure}

\begin{theorem}
\label{theoremS2}
\begin{equation}
\forall n\in\No,\;\displaystyle\sum_{k=0}^{\lfloor n/2 \rfloor}{\displaystyle\sum_{q=0}^{n-2k}{\binom{n-k}q}} = 2^{n+1} - F_{n+2}.
\end{equation}
\end{theorem}
\begin{proof}
We seek to prove that
$$\forall n\in\No,\; S^{[2]}_n\left( 2,-1\right) = 2^{n+1}-F_{n+2}.$$

We have $S^{[2]}_0\left( 2,-1\right)=1=2^{0+1}-F_{0+2}$, so the formula is valid for $n=0$.

Let $n\in\Na$. Then $S^{[2]}_n\left( 2,-1\right)-2S^{[2]}_{n-1}\left( 2,-1\right)=\displaystyle\sum_{k=0}^{\lfloor n/2 \rfloor}B_{n-k,n-2k}-2\!\!\sum_{k=0}^{\lfloor (n-1)/2 \rfloor}B_{n-1-k,n-1-2k}$.
Assume an even $n$, e.g., $n=2s$. Then
\begin{eqnarray*}
S^{[2]}_n\left( 2,-1\right)-2S^{[2]}_{n-1}\left( 2,-1\right)&=&\sum_{k=0}^sB_{2s-k,2s-2k}-2\sum_{k=0}^{s-1}B_{2s-1-k,2s-1-2k}\\
&=&B_{s,0}+\displaystyle\sum_{k=0}^{s-1}(B_{2s-k,2s-2k}-2B_{2s-1-k,2s-1-2k}).
\end{eqnarray*}

From Relation~(\ref{relrecBnk}) and Definition~\ref{def_Bernoulli}, we have
\begin{equation}
\forall p,q\geqslant1,\;B_{p,q}-2B_{p-1,q-1}=B_{p-1,q}-B_{p-1,q-1}=\binom{p-1}q.
\label{relB2diff}
\end{equation}
Therefore
\begin{eqnarray*}
S^{[2]}_n\left( 2,-1\right)-2S^{[2]}_{n-1}\left( 2,-1\right)&=&1+\displaystyle\sum_{k=0}^{s-1}\binom{2s-k-1}{2s-2k}\\
&=&1+\sum_{k=0}^{s-1}\binom{2s-k-1}{k-1}\\
&=&\sum_{k=0}^{s-1}\binom{2s-k-2}{k}.
\end{eqnarray*}
Given that $\forall n\in\No,\;\displaystyle\sum_{k=0}^{\left\lfloor n/2\right\rfloor}\binom{n-k}{k}=F_{n+1}$~\cite[Id.\ 4]{benjaminquinn2003}, we have $$S^{[2]}_n\left( 2,-1\right)-2S^{[2]}_{n-1}\left( 2,-1\right)=F_{(2s-2)+1}=F_{n-1}.$$

The above relation, illustrated by Figure~\ref{figdiagonalesS2}, can be proven for odd values of $n$ using the same method. Applying Lemma~\ref{lemma} to $u_n=S^{[2]}_n\left( 2,-1\right)$ and using Relation~(\ref{relsumfibo}), we obtain $$\forall n\in\No,\;S^{[2]}_n\left( 2,-1\right)=2^n+\displaystyle\sum_{k=1}^{n}2^{n-k}F_{k-1}=2^{n+1}-F_{n+2}.$$
\end{proof}

\begin{figure}[!b]
\begin{center}
{\scalefont{0.9}
\begin{tikzpicture}[baseline=-\the\dimexpr\fontdimen22\textfont2\relax ]
\matrix[matrix of math nodes,row sep=5pt,column sep = 5pt,
 column 11/.style={anchor=base west}, column 12/.style={anchor=base west},] (m) {
1&&&&&&&&&&1=1&\\
1&2&&&&&&&&&2=2&=2\cdot1+F_0\\
1&3&4&&&&&&&&5=4+1&=2\cdot2+F_1\\
1&4&7&8&&&&&&&11=8+3&=2\cdot5+F_2\\
1&5&11&15&16&&&&&&24=16+7+1&=2\cdot11+F_3\\
1&6&16&26&31&32&&&&&51=32+15+4&=2\cdot24+F_4\\
1&7&22&42&57&63&64&&&&107=64+31+11+1&=2\cdot51+F_5\\
1&8&29&64&99&120&127&128&&&222=128+63+26+5&=2\cdot107+F_6\\
};

\begin{pgfonlayer}{myback}
\foreach \element in {m-3-2,m-4-4,m-4-11}{
\fhighlight[blue!20]{\element}{\element}
}
\foreach \element in {m-3-1,m-4-3,m-5-5,m-5-11}{
\fhighlight[red!20]{\element}{\element}
}
\foreach \element in {m-4-2,m-5-4,m-6-6,m-6-11}{
\fhighlight[green!20]{\element}{\element}
}
\foreach \element in {m-4-1,m-5-3,m-6-5,m-7-7,m-7-11}{
\fhighlight[orange!20]{\element}{\element}
}
\end{pgfonlayer}
\end{tikzpicture}
}
\captionof{figure}{Fibonacci numbers resulting from $S^{[2]}_n\left( 2,-1\right)-2S^{[2]}_{n-1}\left( 2,-1\right)$, forming \seqnum{A027934}~\cite{oeis}.}
\label{figdiagonalesS2}
\end{center}
\end{figure}

\subsection{The path following direction \texorpdfstring{$(c,1-c)$}{(c,1-c)} from \texorpdfstring{$(n,n)$}{(n,n)}}
\begin{definition}
For $c \geqslant 2,\;n\geqslant1$, let $\lambda_n(c)=S^{[2]}_n\left(c,1-c\right)-2S^{[2]}_{n-1}\left(c,1-c\right)$.
\end{definition}
The sequence $\left(\lambda_n(c)\right)_{n\in\Na}$ corresponds to \seqnum{A000930} for $c = 3$ (see Figure~\ref{figdiagonaleslambda3}), \seqnum{A003269} for $c = 4$ and \seqnum{A003520} for $c = 5$~\cite{oeis}. We observe that $\left(\lambda_n(c)\right)_{n\in\Na}$ satisfies the following linear recurrence relation.
\begin{theorem}
$$\forall n\in\llbracket1,c-1\rrbracket,\;\lambda_n(c)=0;\;\lambda_{c}(c)=1\text{ and }\forall n>c,\; \lambda_n(c)=\lambda_{n-1}(c)+\lambda_{n-c}(c).$$
\label{theoremlambda}
\end{theorem}
The sequence $\left(\lambda_n(c)\right)_{n\in\Na}$ is a generalization the Fibonacci sequence, which corresponds to the $c=2$ case (as shown above, $\forall n\in\Na, \lambda_n(2) = F_{n-1}$).

\begin{figure}[!b]
\begin{center}
{\scalefont{0.9}
\begin{tikzpicture}[baseline=-\the\dimexpr\fontdimen22\textfont2\relax ]

\matrix[matrix of math nodes,row sep=5pt,column sep = 5pt,
 column 11/.style={anchor=base west}, column 12/.style={anchor=base west},] (m) {
1&&&&&&&&&&1=1&\\
1&2&&&&&&&&&2=2&=2\cdot1+0\\
1&3&4&&&&&&&&4=4&=2\cdot2+0\\
1&4&7&8&&&&&&&9=8+1&=2\cdot4+1\\
1&5&11&15&16&&&&&&19=16+3&=2\cdot9+1\\
1&6&16&26&31&32&&&&&39=32+7&=2\cdot19+1\\
1&7&22&42&57&63&64&&&&80=64+15+1&=2\cdot39+2\\
1&8&29&64&99&120&127&128&&&163=128+31+4&=2\cdot80+3\\
};

\begin{pgfonlayer}{myback}
\foreach \element in {m-2-1,m-4-4,m-4-11}{
\fhighlight[blue!20]{\element}{\element}
}
\foreach \element in {m-3-2,m-5-5,m-5-11}{
\fhighlight[red!20]{\element}{\element}
}
\foreach \element in {m-4-3,m-6-6,m-6-11}{
\fhighlight[green!20]{\element}{\element}
}
\foreach \element in {m-3-1,m-5-4,m-7-7,m-7-11}{
\fhighlight[orange!20]{\element}{\element}
}
\end{pgfonlayer}

\end{tikzpicture}
}
\captionof{figure}{Instances of $S^{[2]}_n\left(3,-2\right)$ in $\mathcal B^{[2]}$ form \seqnum{A099568}~\cite{oeis} and the difference $S^{[2]}_n\left(3,-2\right)-2S^{[2]}_{n-1}\left(3,-2\right)$ yields the terms of $(\lambda_n(3))_{n\in\Na}$, corresponding to \seqnum{A000930}~\cite{oeis}.}
\label{figdiagonaleslambda3}
\end{center}
\end{figure}

\begin{proof}
Let $c\geqslant2$, $n\geqslant1$. From Definition~\ref{defpathsST}, we have
$$\lambda_n(c)=\sum_{k=0}^{\left\lfloor n/c\right\rfloor}B_{n+k(1-c),n-kc}-2\sum_{k=0}^{\left\lfloor (n-1)/c\right\rfloor}B_{n-1+k(1-c),n-1-kc}.$$

We first ascertain the initial values of $(\lambda_n(c))_{n\in\Na}$.
For $n<c$, $\lfloor n/c\rfloor=\lfloor (n-1)/c\rfloor=0$ and $\lambda_n(c)=B_{n,n}-2B_{n-1,n-1}=2^n-2\cdot 2^{n-1}=0$.
For $n=c$, $\lambda_c(c)=B_{c,c}+B_{1,0}-2B_{c-1,c-1}=2^c+1-2\cdot 2^{c-1}=1$.

We then prove the recurrence relation for $n>c$. We have
$$\lambda_{n-1}(c)=\sum_{k=0}^{\left\lfloor (n-1)/c\right\rfloor}B_{n-1+k(1-c),n-1-kc}-2\sum_{k=0}^{\left\lfloor (n-2)/c\right\rfloor}B_{n-2+k(1-c),n-2-kc}$$
and
$$\lambda_{n-c}(c)=\sum_{k=0}^{\left\lfloor (n-c)/c\right\rfloor}B_{n-c+k(1-c),n-c-kc}-2\sum_{k=0}^{\left\lfloor (n-c-1)/c\right\rfloor}B_{n-c-1+k(1-c),n-c-1-kc}.$$
The summation upper bounds of $\lambda_{n}(c)$, $\lambda_{n-1}(c)$ and $\lambda_{n-c}(c)$ depend on the remainder of the Euclidean division of $n$ by $c$. Let $n=cq+r>c$ with $q=\left\lfloor n/c\right\rfloor\in\Na, r\in\llbracket0,c-1\rrbracket$. There are three cases to consider, which are shown in Table~\ref{tblremainder}.
\begin{table}[!h]
\begin{center}
$\begin{array}{c|ccccc}
&\left\lfloor n/c\right\rfloor&\left\lfloor (n-1)/c\right\rfloor&\left\lfloor (n-2)/c\right\rfloor&\left\lfloor (n-c)/c\right\rfloor&\left\lfloor (n-c-1)/c\right\rfloor\\\hline
r=0&q&q-1&q-1&q-1&q-2\\
r=1&q&q&q-1&q-1&q-1\\
1<r<c&q&q&q&q-1&q-1
\end{array}$
\end{center}
\caption{Summation upper bounds for the various values of $r=n-c\left\lfloor n/c\right\rfloor$.}
\label{tblremainder}
\end{table}

For $r=0$, we have
\begin{eqnarray*}
\lambda_n(c)&=&\sum_{k=0}^{q}B_{cq+k(1-c),cq-kc}-2\sum_{k=0}^{q-1}B_{cq-1+k(1-c),cq-1-kc}\\
&=&B_{cq+q(1-c),cq-qc}+\sum_{k=0}^{q-1}\left(B_{cq+k(1-c),cq-kc}-2B_{cq-1+k(1-c),cq-1-kc}\right),
\end{eqnarray*}
and Relation~(\ref{relB2diff}) yields
\begin{eqnarray*}
\lambda_n(c)&=&B_{q,0}+\sum_{k=0}^{q-1}\binom{cq-1+k(1-c)}{cq-kc}\\
&=&1+\sum_{k=0}^{q-1}\binom{cq-1+k(1-c)}{k-1}\\
&=&1+\sum_{k=1}^{q-1}\binom{cq-1+k(1-c)}{k-1}.
\end{eqnarray*}
Similarly,
\begin{eqnarray*}
\lambda_{n-1}(c)&=&\sum_{k=0}^{q-1}B_{cq-1+k(1-c),cq-1-kc}-2\sum_{k=0}^{q-1}B_{cq-2+k(1-c),cq-2-kc}\\
&=&\sum_{k=0}^{q-1}\binom{cq-2+k(1-c)}{k-1},
\end{eqnarray*}
and
\begin{eqnarray*}
\lambda_{n-c}(c)&=&\sum_{k=0}^{q-1}B_{cq-c+k(1-c),cq-c-kc}-2\sum_{k=0}^{q-2}B_{cq-c-1+k(1-c),cq-c-1-kc}\\
&=&\underbrace{B_{cq-c+(q-1)(1-c),cq-c-(q-1)c}}_{B_{q-1,0}}+\sum_{k=0}^{q-2}(\underbrace{B_{cq-c+k(1-c),cq-c-kc}-2B_{cq-c-1+k(1-c),cq-c-1-kc}}_{\binom{cq-c-1+k(1-c)}{cq-c-kc}})\\
&=&1+\sum_{k=0}^{q-2}\binom{cq-c-1+k(1-c)}{k-1}\\
&=&1+\sum_{k=1}^{q-1}\binom{cq-2+k(1-c)}{k-2}.
\end{eqnarray*}
Therefore we have
\begin{eqnarray*}
\lambda_{n-1}(c)+\lambda_{n-c}(c)&=&0+\sum_{k=1}^{q-1}\binom{cq-2+k(1-c)}{k-1}+\sum_{k=1}^{q-1}\binom{cq-2+k(1-c)}{k-2}+1\\
&=&\sum_{k=1}^{q-1}\binom{cq-1+k(1-c)}{k-1}+1\\
&=&\lambda_n(c).
\end{eqnarray*}

This proves Theorem~\ref{theoremlambda} for the $r=0$ case. We can use the same method for $r=1$ or $r>1$, since only the summation upper bounds are modified in those cases.
\end{proof}

We can write the explicit expression of $\lambda_n(c)$ as $\displaystyle\lambda_n(c)=\sum_{i=0}^{\left\lfloor (n-c)/(c-1)\right\rfloor}\binom{n-c+i(1-c)}{i}.$ We obtain, using Lemma~\ref{lemma}, $S^{[2]}_n(c,1-c)=2^n+\displaystyle\sum_{k=1}^n2^{n-k}\lambda_k(c)$, which leads to the following relation.
\begin{corollary}
$$\displaystyle\sum_{k=0}^{\lfloor n/c \rfloor}{\displaystyle\sum_{q=0}^{n-kc}{\binom{n-(c-1)k}q}}=2^n+\displaystyle\sum_{k=1}^n2^{n-k}\displaystyle\sum_{i=0}^{\left\lfloor (k-c)/(c-1)\right\rfloor}\binom{k-c-(c-1)i}{i}.$$
\end{corollary}

\subsection{The path following direction \texorpdfstring{$(-1,-1)$}{(-1,-1)} from \texorpdfstring{$(n,0)$}{(n,0)}}
\label{secT2}
The path following direction $(-1,-1)$ from $(n,0)$ has a connection to the Fibonacci sequence that appears in the difference between successive terms of $T^{[2]}_n(-1,-1)$, as illustrated by Figure~\ref{figdiagonalesT2}. Not only do we notice that
$$\forall p\geqslant1,\;T^{[2]}_{2p}(-1,-1)=T^{[2]}_{2p-1}(-1,-1)+F_{2p+1},$$
but also that
$$\forall p\geqslant1,\;T^{[2]}_{2p+1}(-1,-1)=T_{2p}^{[2]}(-1,-1)+T^{[2]}_{2p-1}(-1,-1).$$
\begin{proof}
To prove these two recurrence relations, we rely on Relation~(\ref{relrecBnk}). For odd indices, we have
\begin{eqnarray*}
\forall p\geqslant1,\;
T^{[2]}_{2p+1}(-1,-1)&=&\displaystyle\sum_{k=0}^{p}B_{2p+1-k,k}\\
&=&B_{2p+1,0}+\displaystyle\sum_{k=1}^{p}\left(B_{2p-k,k-1}+B_{2p-k,k}\right)\\
&=&B_{2p,0}+\displaystyle\sum_{k=1}^{p}B_{2p-k,k}+\displaystyle\sum_{k=0}^{p-1}B_{2p-1-k,k}\\
&=&\displaystyle\sum_{k=0}^{p}B_{2p-k,k}+\displaystyle\sum_{k=0}^{p-1}B_{2p-1-k,k}\\
&=&T^{[2]}_{2p}(-1,-1)+T^{[2]}_{2p-1}(-1,-1).
\end{eqnarray*}
For even indices,
\begin{eqnarray*}
\forall p\geqslant1,\;T^{[2]}_{2p}(-1,-1)&=&\displaystyle\sum_{k=0}^{p}B_{2p-k,k}\\
&=&\displaystyle\sum_{k=0}^{p}\displaystyle\sum_{q=0}^{k}\binom{2p-k}{q}\\
&=&\displaystyle\sum_{k=0}^{p}\left(\displaystyle\sum_{q=0}^{k-1}\binom{2p-k}{q}+\binom{2p-k}{k}\right)\\
&=&\displaystyle\sum_{k=0}^{p-1}\displaystyle\sum_{q=0}^{k}\binom{2p-1-k}{q}+\displaystyle\sum_{k=0}^{p}\binom{2p-k}{k}.
\end{eqnarray*}
Since $\forall n\in\No,\;\displaystyle\sum_{k=0}^{\left\lfloor n/2\right\rfloor}\binom{n-k}{k}=F_{n+1}$~\cite[Id.\ 4]{benjaminquinn2003}, we have $\displaystyle\sum_{k=0}^{p}\binom{2p-k}{k}=F_{2p+1}$. Hence,
$$\forall p\geqslant1,\;T^{[2]}_{2p}(-1,-1)=T^{[2]}_{2p-1}(-1,-1)+F_{2p+1}.$$
\end{proof}

These two relations make it possible to derive a recurrence relation that pertains only to the odd subsequence:
$$\forall p\geqslant1,\;T^{[2]}_{2p+1}(-1,-1)-2T^{[2]}_{2p-1}(-1,-1)=F_{2p+1}.$$
Therefore, using Lemma~\ref{lemma} and Relation~(\ref{relsumfibo2}), we have
\begin{eqnarray*}
\forall p\geqslant1,\;T^{[2]}_{2p+1}(-1,-1)&=&2^p\underbrace{T^{[2]}_{1}(-1,-1)}_1+\displaystyle\sum_{k=1}^p2^{p-k}F_{2k+1}\\
&=&2^p+F_{2p+4}-3\cdot 2^p\\
&=&F_{2p+4}-2^{p+1}.
\end{eqnarray*}
We feed back the odd subsequence into the expression for the even subsequence as follows:
\begin{eqnarray*}
\forall p\geqslant1,\;T^{[2]}_{2p}(-1,-1)&=&T^{[2]}_{2p-1}(-1,-1)+F_{2p+1}\\
&=&F_{2p+2}-2^{p}+F_{2p+1}\\
&=&F_{2p+3}-2^p.
\end{eqnarray*}

We have thus proven that $\forall n\in\No,\;T^{[2]}_{n}(-1,-1)=F_{n+3}-2^{\left\lfloor(n+1)/2\right\rfloor}$, which we can rewrite as the following theorem.
\begin{theorem}
\label{resT2}
$$\forall n\in\No,\;\displaystyle\sum_{k=0}^{\left\lfloor n/2\right\rfloor}\sum_{q=0}^k\binom{n-k}q=F_{n+3}-2^{\left\lfloor(n+1)/2\right\rfloor}.$$
\end{theorem}

\begin{figure}[!t]
\begin{center}
{\scalefont{0.9}
\begin{tikzpicture}[baseline=-\the\dimexpr\fontdimen22\textfont2\relax ]
\matrix[matrix of math nodes,row sep=5pt,column sep = 5pt,
 column 11/.style={anchor=base west}, column 12/.style={anchor=base west},] (m) {
1&&&&&&&&&&1=1&\\
1&2&&&&&&&&&1=1&\\
1&3&4&&&&&&&&3=1+2&=1+F_3\\
1&4&7&8&&&&&&&4=1+3&=3+1\\
1&5&11&15&16&&&&&&9=1+4+4&=4+F_5\\
1&6&16&26&31&32&&&&&13=1+5+7&=9+4\\
1&7&22&42&57&63&64&&&&26=1+6+11+8&=13+F_7\\
1&8&29&64&99&120&127&128&&&39=1+7+16+15&=26+13\\
1&9&37&93&163&219&247&255&256&&73=1+8+22+26+16&=39+F_9\\
};

\begin{pgfonlayer}{myback}
\foreach \element in {m-3-2,m-4-1,m-4-11}{
\fhighlight[blue!20]{\element}{\element}
}
\foreach \element in {m-3-3,m-4-2,m-5-1,m-5-11}{
\fhighlight[red!20]{\element}{\element}
}
\foreach \element in {m-4-3,m-5-2,m-6-1,m-6-11}{
\fhighlight[green!20]{\element}{\element}
}
\foreach \element in {m-4-4,m-5-3,m-6-2,m-7-1,m-7-11}{
\fhighlight[orange!20]{\element}{\element}
}
\end{pgfonlayer}
\end{tikzpicture}
}
\captionof{figure}{Instances of $T^{[2]}_n\left(-1,-1\right)$ in Bernoulli's triangle and differences from the preceding terms.}
\label{figdiagonalesT2}
\end{center}
\end{figure}

\begin{remark}
Among the paths of Bernoulli's triangle following other directions, one can find other sequences that follow linear recurrence relations. For example, $\left(\bar S^{[2]}_n(3,-1)\right)_{n\in\No}$ corresponds to \seqnum{A005251}~\cite{oeis} and $\left(\bar S^{[2]}_n(4,-1)\right)_{n\in\No}$ to \seqnum{A138653}~\cite{oeis}. We believe that further sequences related to partial sums of binomial coefficients are yet to be uncovered in $\mathcal B^{[2]}$.
\end{remark}

\section{Bernoulli's third-order triangle}
We now consider Bernoulli's third-order triangle $\mathcal B^{[3]}$ (see Figure~\ref{figtriangleB3}). From Definition~\ref{defBm},
$$\forall n\in\No,k\in\llbracket0,n\rrbracket,\;B^{[3]}_{n,k}=\displaystyle\sum_{q=0}^kB_{n,q}=\displaystyle\sum_{q=0}^k\displaystyle\sum_{r=0}^q\binom nr.$$
Recall that $B^{[3]}_{n,0}=1$ and $\forall n\geqslant2,\;\forall k\in\llbracket1,n-1\rrbracket,\;B^{[3]}_{n,k}=B^{[3]}_{n-1,k}+B^{[3]}_{n-1,k-1}$. Furthermore, $B^{[3]}_{n,n}=(n+2)2^{n-1}$.
\begin{figure}[!h]
\begin{center}
{\scalefont{0.9}
\begin{tikzpicture}
\matrix[matrix of math nodes,row sep=2pt,column sep = 2pt] (m) {
&k&0&1&2&3&4&5&6&7&8&9\\
n\\
0&&1\\
1&&1&3\\
2&&1&4&8\\
3&&1&5&12&20\\
4&&1&6&17&32&48\\
5&&1&7&23&49&80&112\\
6&&1&8&30&72&129&192&256\\
7&&1&9&38&102&201&321&448&576\\
8&&1&10&47&140&303&522&769&1024&1280\\
9&&1&11&57&187&443&825&1291&1793&2304&2816\\
};
\end{tikzpicture}
}
\captionof{figure}{Bernoulli's third-order triangle (\seqnum{A193605}~\cite{oeis}): the element in row $n$ and column $k$ corresponds to $B^{[3]}_{n,k}$.}
\label{figtriangleB3}
\end{center}
\end{figure}
\subsection{The path following direction \texorpdfstring{$(2,-1)$}{(2,-1)} from \texorpdfstring{$(n,n)$}{(n,n)}}
Definition~\ref{defpathsST} states that $\bar S^{[3]}_n\left(2,-1\right)=B^{[3]}_{n,n} - \displaystyle\sum_{k=1}^{\lfloor n/2 \rfloor}{B^{[3]}_{n-k,n-2k}}$. Figure~\ref{figdiagonalesS3} suggests the following recurrence relation:
\begin{equation}
\label{relrecS3}
\bar S^{[3]}_n\left(2,-1\right)-2\bar S^{[3]}_{n-1}\left(2,-1\right)=F_n.
\end{equation}

\begin{figure}[!t]
\begin{center}
{\scalefont{0.9}
\begin{tikzpicture}[baseline=-\the\dimexpr\fontdimen22\textfont2\relax]

\matrix[matrix of math nodes,row sep=5pt,column sep = 5pt,
 column 12/.style={anchor=base west},] (m) {
1&&&&&&&&&&\quad&1=1\\
1&3&&&&&&&&&\quad&3=3&=2\cdot1+F_1\\
1&4&8&&&&&&&&\quad&7=8-1&=2\cdot3+F_2\\
1&5&12&20&&&&&&&\quad&16=20-4&=2\cdot7+F_3\\
1&6&17&32&48&&&&&&\quad&35=48-12-1&=2\cdot16+F_4\\
1&7&23&49&80&112&&&&&\quad&75=112-32-5&=2\cdot35+F_5\\
1&8&30&72&129&192&256&&&&\quad&158=256-80-17-1&=2\cdot75+F_6\\
1&9&38&102&201&321&448&576&&&\quad&329=576-192-49-6&=2\cdot158+F_7\\
};
\begin{pgfonlayer}{myback}
\foreach \element in {m-3-2,m-4-4,m-4-12}{
\fhighlight[blue!20]{\element}{\element}
}
\foreach \element in {m-3-1,m-4-3,m-5-5,m-5-12}{
\fhighlight[red!20]{\element}{\element}
}
\foreach \element in {m-4-2,m-5-4,m-6-6,m-6-12}{
\fhighlight[green!20]{\element}{\element}
}
\foreach \element in {m-4-1,m-5-3,m-6-5,m-7-7,m-7-12}{
\fhighlight[orange!20]{\element}{\element}
}
\end{pgfonlayer}

\end{tikzpicture}
}
\captionof{figure}{Instances of $\bar S^{[3]}_n\left(2,-1\right)$ in $\mathcal B^{[3]}$ that illustrate Relation~(\ref{relrecS3}).}
\label{figdiagonalesS3}
\end{center}
\end{figure}

\begin{proof}
Let $n\in\Na$. We have $$\bar
S^{[3]}_n\left(2,-1\right)-2\bar S^{[3]}_{n-1}\left(2,-1\right)=\displaystyle B^{[3]}_{n,n}-\sum_{k=1}^{\lfloor n/2\rfloor}B^{[3]}_{n-k,n-2k}-2B^{[3]}_{n-1,n-1}+2\sum_{k=1}^{\lfloor (n-1)/2\rfloor}B^{[3]}_{n-1-k,n-1-2k}.$$

Firstly, $B^{[3]}_{n,n}-2B^{[3]}_{n-1,n-1}=(n+2)2^{n-1}-2(n+1)2^{n-2}=2^{n-1}$.

Let us assume that $n$ is even, e.g., $n=2s$ (the method is the same for odd values of $n$). We have
\begin{eqnarray*}
\bar S^{[3]}_n\left(2,-1\right)-2\bar S^{[3]}_{n-1}\left(2,-1\right)&=&2^{2s-1}-\sum_{k=1}^{s}B^{[3]}_{2s-k,2s-2k}+2\sum_{k=1}^{s-1}B^{[3]}_{2s-1-k,2s-1-2k}\\
&=&2^{2s-1}-\underbrace{B^{[3]}_{s,0}}_1-\sum_{k=1}^{s-1}\left(B^{[3]}_{2s-k,2s-2k}-2B^{[3]}_{2s-1-k,2s-1-2k}\right).\\
\end{eqnarray*}

From Definition~\ref{defBm} and Relation~(\ref{relrecBm}), we have, for $m\geqslant2$,
$$\forall p,q\geqslant1,\;B^{[m]}_{p,q}-2B^{[m]}_{p-1,q-1}=B^{[m]}_{p-1,q}-B^{[m]}_{p-1,q-1}=B^{[m-1]}_{p-1,q},$$
hence,
$$\bar S^{[3]}_n\left(2,-1\right)-2\bar S^{[3]}_{n-1}\left(2,-1\right)=2^{2s-1}-1-\sum_{k=1}^{s-1}B_{2s-1-k,2s-2k}.$$

Moreover,
\begin{eqnarray*}
\sum_{k=1}^{s-1}B_{2s-1-k,2s-2k}&=&\sum_{k=2}^{s}B_{2s-k,2s-2k+2}\\
&=&\sum_{k=2}^{s}\left(B_{2s-k,2s-2k}+\binom{2s-k}{2s-2k+1}+\binom{2s-k}{2s-2k+2}\right)\\
&=&\sum_{k=2}^{s}B_{2s-k,2s-2k}+\sum_{k=2}^{s}\binom{2s-k+1}{2s-2k+2}\\
&=&S^{[2]}_{2s}\left(2,-1\right)-\underbrace{B_{2s,2s}}_{2^{2s}}-B_{2s-1,2s-2}+\sum_{k=2}^{s}\binom{2s-k+1}{k-1}.
\end{eqnarray*}

Theorem~\ref{theoremS2} states that $S^{[2]}_{2s}\left(2,-1\right)=2^{2s+1}-F_{2s+2}$.

We also have $B_{2s-1,2s-2}=B_{2s-1,2s-1}-\displaystyle\binom{2s-1}{2s-1}=2^{2s-1}-1$. Furthermore, \begin{eqnarray*}\displaystyle\sum_{k=2}^{s}\binom{2s+1-k}{k-1}
&=&\sum_{k=1}^{s}\binom{2s-k+1}{k-1}-\binom{2s}0\\
&=&\sum_{k=0}^{s-1}\binom{2s-k}{k}-1\\
&=&\sum_{k=0}^{s}\binom{2s-k}{k}-\binom ss-1\\
&=&F_{2s+1}-2.
\end{eqnarray*}

Finally,
\begin{eqnarray*}
\bar S^{[3]}_n\left(2,-1\right)-2\bar S^{[3]}_{n-1}\left(2,-1\right)&=&2^{2s-1}-1-\left(2^{2s+1}-F_{2s+2}-2^{2s}-(2^{2s-1}-1)+F_{2s+1}-2\right)\\
&=&F_{2s+2}-F_{2s+1}\\
&=&F_{2s}\\
&=&F_{n}.
\end{eqnarray*}
\end{proof}

Lemma~\ref{lemma} and Relation~(\ref{relsumfibo}) lead to an explicit expression of $\bar S^{[3]}_n\left(2,-1\right)$, as follows:
\begin{eqnarray*}
\bar S^{[3]}_n\left(2,-1\right)&=&2^n\bar S^{[3]}_0\left(2,-1\right)+\displaystyle\sum_{k=1}^n2^{n-k}F_k\\
&=&2^n+\displaystyle\sum_{k=1}^{n+1}2^{n+1-k}F_{k-1}\\
&=&2^n+\left(2^{n+1}-F_{n+3}\right)\\
&=&3\cdot 2^n-F_{n+3}.
\end{eqnarray*}

From the explicit form of $\bar S^{[3]}_n\left(2,-1\right)$, we deduce that $S^{[3]}_n\left(2,-1\right)=2B^{[3]}_{n,n}+F_{n+3}-3\cdot 2^n$ and obtain the following theorem.
\begin{theorem}
\label{theoremS3}
$$\displaystyle\sum_{k=0}^{\lfloor n/2 \rfloor}\sum_{q=0}^{n-2k}\sum_{r=0}^{q}\binom {n-k}r = F_{n+3}+(n-1)2^n.$$
\end{theorem}

\subsection{The path following direction \texorpdfstring{$(-1,-1)$}{(-1,-1)} from \texorpdfstring{$(n,0)$}{(n,0)}}

In $\mathcal B^{[3]}$, the path following direction $(-1,-1)$ from $(n,0)$ does not appear at first glance to have any obvious connection to the Fibonacci sequence; however, the differences of consecutive terms of $T^{[3]}_n(-1,-1)$ follow a pattern similar to $T^{[2]}_n(-1,-1)$. Indeed, the odd subsequence of the difference sequence has the same behavior as in $\mathcal B^{[2]}$, while the even subsequence has a connection to the $T^{[2]}_{n}(-1,-1)$ numbers of even $n$ (see Figure~\ref{figdiagonalesT3}).

More precisely, we notice that $\forall p\geqslant1,\;T^{[3]}_{2p}(-1,-1)=T^{[3]}_{2p-1}(-1,-1)+T^{[2]}_{2p}(-1,-1)$ and $T^{[3]}_{2p+1}(-1,-1)=T^{[3]}_{2p}(-1,-1)+T^{[3]}_{2p-1}(-1,-1)$.

\begin{figure}[!t]
\begin{center}
{\scalefont{0.9}
\begin{tikzpicture}[font=\small, baseline=-\the\dimexpr\fontdimen22\textfont2\relax ]

\matrix[matrix of math nodes,row sep=5pt,column sep = 5pt,
 column 11/.style={anchor=base west}, column 12/.style={anchor=base west},] (m) {
1&&&&&&&&&&1=1&\\
1&3&&&&&&&&&1=1&\\
1&4&8&&&&&&&&4=1+3&=1+T^{[2]}_2(-1,-1)\\
1&5&12&20&&&&&&&5=1+4&=4+1\\
1&6&17&32&48&&&&&&14=1+5+8&=5+T^{[2]}_4(-1,-1)\\
1&7&23&49&80&112&&&&&19=1+6+12&=14+5\\
1&8&30&72&129&192&256&&&&45=1+7+17+20&=19+T^{[2]}_6(-1,-1)\\
1&9&38&102&201&321&448&576&&&64=1+8+23+32&=45+19\\
};
\begin{pgfonlayer}{myback}\foreach \element in {m-3-2,m-4-1,m-4-11}{
\fhighlight[blue!20]{\element}{\element}
}
\foreach \element in {m-3-3,m-4-2,m-5-1,m-5-11}{
\fhighlight[red!20]{\element}{\element}
}
\foreach \element in {m-4-3,m-5-2,m-6-1,m-6-11}{
\fhighlight[green!20]{\element}{\element}
}
\foreach \element in {m-4-4,m-5-3,m-6-2,m-7-1,m-7-11}{
\fhighlight[orange!20]{\element}{\element}
}
\end{pgfonlayer}
\end{tikzpicture}
}
\captionof{figure}{Differences between successive terms of $T^{[3]}_n\left(-1,-1\right)$ in Bernoulli's third-order triangle.}
\label{figdiagonalesT3}
\end{center}
\end{figure}

\begin{proof}
We prove these two relations directly at order $m$, using Relation~(\ref{relrecBm}) and proceeding as in the derivation of Theorem~\ref{resT2} (section~\ref{secT2}). For odd terms, we have
\begin{eqnarray}
\forall m\geqslant2,\forall p\geqslant1,\;
T^{[m]}_{2p+1}(-1,-1)&=&\displaystyle\sum_{k=0}^{p}B^{[m]}_{2p+1-k,k}\nonumber\\
&=&B^{[m]}_{2p+1,0}+\displaystyle\sum_{k=1}^{p}\left(B^{[m]}_{2p-k,k-1}+B^{[m]}_{2p-k,k}\right)\nonumber\\
&=&B^{[m]}_{2p,0}+\displaystyle\sum_{k=1}^{p}B^{[m]}_{2p-k,k}+\displaystyle\sum_{k=0}^{p-1}B^{[m]}_{2p-1-k,k}\nonumber\\
&=&\displaystyle\sum_{k=0}^{p}B^{[m]}_{2p-k,k}+\displaystyle\sum_{k=0}^{p-1}B^{[m]}_{2p-1-k,k},\nonumber\\
T^{[m]}_{2p+1}(-1,-1)&=&T^{[m]}_{2p}(-1,-1)+T^{[m]}_{2p-1}(-1,-1).
\label{relrecTmodd}
\end{eqnarray}
For even terms,
\begin{eqnarray}
\forall m\geqslant2,\forall p\geqslant1,\;T^{[m]}_{2p}(-1,-1)&=&\displaystyle\sum_{k=0}^{p}B^{[m]}_{2p-k,k}\nonumber\\
&=&\displaystyle\sum_{k=0}^{p}\displaystyle\sum_{q=0}^{k}B^{[m-1]}_{2p-k,q}\nonumber\\
&=&\displaystyle\sum_{k=0}^{p}\left(\displaystyle\sum_{q=0}^{k-1}B^{[m-1]}_{2p-k,q}+B^{[m-1]}_{2p-k,k}\right)\nonumber\\
&=&\displaystyle\sum_{k=0}^{p-1}\displaystyle\sum_{q=0}^{k}B^{[m-1]}_{2p-1-k,q}+\displaystyle\sum_{k=0}^{p}B^{[m-1]}_{2p-k,k}\nonumber\\
&=&\displaystyle\sum_{k=0}^{p-1}B^{[m]}_{2p-1-k,k}+T^{[m-1]}_{2p}(-1,-1),\nonumber\\
T^{[m]}_{2p}(-1,-1)&=&T^{[m]}_{2p-1}(-1,-1)+T^{[m-1]}_{2p}(-1,-1).
\label{relrecTmeven}
\end{eqnarray}
\end{proof}

From Relations~(\ref{relrecTmodd}) and~(\ref{relrecTmeven}), we derive a general expression of $T^{[3]}_n(-1,-1)$ for $n\in\No$. Isolating the odd subsequence leads to
$$\forall p\geqslant1,\;T^{[3]}_{2p+1}(-1,-1)-2T^{[3]}_{2p-1}(-1,-1)=T^{[2]}_{2p}(-1,-1).$$

Therefore, using Lemma~\ref{lemma} and Relation~(\ref{relsumfibo2}), we obtain
\begin{eqnarray*}
\forall p\geqslant1,\;T^{[3]}_{2p+1}(-1,-1)&=&2^p\underbrace{T^{[3]}_{1}(-1,-1)}_1+\displaystyle\sum_{k=1}^p2^{p-k}\underbrace{T^{[2]}_{2k}(-1,-1)}_{F_{2k+3}-2^k}\\
&=&2^p+\displaystyle\sum_{k=1}^p2^{p-k}F_{2k+3}-\sum_{k=1}^p2^p\\
&=&2^p+F_{2p+6}-2^{p+3}-p2^p\\
&=&F_{2p+6}-(p+7)2^p.
\end{eqnarray*}
Hence,
\begin{eqnarray*}
\forall p\geqslant1,\;T^{[3]}_{2p}(-1,-1)&=&T^{[3]}_{2p-1}(-1,-1)+T^{[2]}_{2p}(-1,-1)\\
&=&F_{2p+4}-(p+6)2^{p-1}+F_{2p+3}-2^p\\
&=&F_{2p+5}-(p+8)2^{p-1}.
\end{eqnarray*}

Finally, we reach the following relation:
$$\forall n\in\No,\;T^{[3]}_{n}(-1,-1)=F_{n+5}-2^{\left\lfloor(n+1)/2\right\rfloor}\left(\frac12\left\lfloor(n+1)/2\right\rfloor+\frac72+\frac{(-1)^n}2\right).$$
\begin{theorem}
\label{resT3}
$$\forall n\in\No,\;\displaystyle\sum_{k=0}^{\left\lfloor n/2\right\rfloor}\sum_{q=0}^k\sum_{r=0}^q\binom{n-k}r=F_{n+5}-2^p\left(\frac12p+\frac72+\frac{(-1)^n}2\right),$$
where $p={\left\lfloor(n+1)/2\right\rfloor}$.
\end{theorem}

\begin{remark}
Paths in $\mathcal B^{[3]}$ following other directions yield interesting sequences in a similar manner to those in $\mathcal B^{[2]}$. For example, $\left(\bar S^{[3]}_n(3,-1)-2\bar S^{[3]}_{n-1}(3,-1)\right)_{n\in\Na}$ corresponds to \seqnum{A005314}~\cite{oeis}. Further work should be able to uncover additional sequences by following other paths in $\mathcal B^{[3]}$.
\end{remark}

\section{Higher-order triangles}
\subsection{The path following direction \texorpdfstring{$(-1,-1)$}{(-1,-1)} from \texorpdfstring{$(n,0)$}{(n,0)}}
Using Relations~(\ref{relrecTmodd}) and~(\ref{relrecTmeven}) for $m=4$ and $m=5$, we proceed as we did for $T^{[3]}_{n}(-1,-1)$ to determine the sums over the path following direction $(-1,-1)$ from $(n,0)$ in $\mathcal B^{[4]}$ and $\mathcal B^{[5]}$. We obtain the following identities:
$$\begin{array}{rcl}
\forall n\in\No,\;\displaystyle\sum_{k=0}^{\lfloor n/2 \rfloor}\sum_{q=0}^{k}\sum_{r=0}^{q}\sum_{s=0}^{r}\binom {n-k}s&=&F_{n+7}-2^{p}\left(\dfrac 18p^2+\dfrac{17}8p+10+(-1)^n\left(\dfrac14 p+2\right)\right),\\
\text{and }\displaystyle\sum_{k=0}^{\lfloor n/2 \rfloor}\sum_{q=0}^{k}\sum_{r=0}^{q}\sum_{s=0}^{r}\sum_{t=0}^{s}\binom {n-k}t&=&\\\multicolumn{3}{r}{F_{n+9}-2^{p}\left(\dfrac1{48}{p^3}+\dfrac{5}{8}p^2+\dfrac{317}{48}p+27+(-1)^n\left(\dfrac1{16}p^2+\dfrac{19}{16}p+6\right)\right)},
\end{array}$$
where $p=\left\lfloor\frac{n+1}{2}\right\rfloor$.

The expressions of $T^{[m]}_{n}(-1,-1)$ for $m\in\llbracket2,5\rrbracket$ contain a Fibonacci term and a second term composed of the product of a power of $2$ by a polynomial, which suggests a general connection between $T^{[m]}_{n}(-1,-1)$  and the Fibonacci sequence. We write a general expression of $T^{[m]}_{n}(-1,-1)$ in the following theorem.
\begin{theorem}
$$\forall m\in\Na,\forall n\in\No,\;\displaystyle\sum_{i_1=0}^{\left\lfloor n/2\right\rfloor}\sum_{i_2=0}^{i_1}\cdots\sum_{i_m=0}^{i_{m-1}}\binom{n-i_1}{i_m}=F_{n+2m-1}-2^p\left(Q^{[m]}\left(p\right)+(-1)^nR^{[m]}\left(p\right)\right),$$
where $p=\left\lfloor\frac{n+1}{2}\right\rfloor$, $Q^{[m]}$ and $R^{[m]}$ are polynomials with coefficients in $\mathbb Q$ and of degree $(m-2)^+$ and $(m-3)^+$ respectively.
\label{theoremTm}
\end{theorem}
\begin{proof}
We seek to prove that $\forall m\geqslant1$,$$\forall n\in\No,\;T^{[m]}_{n}(-1,-1)=F_{n+2m-1}-2^p\left(Q^{[m]}\left(p\right)+(-1)^nR^{[m]}\left(p\right)\right),$$where $p=\left\lfloor\frac{n+1}{2}\right\rfloor$, $Q^{[m]}$ and $R^{[m]}$ are polynomials with coefficients in $\mathbb Q$ and of degree $(m-2)^+$ and $(m-3)^+$ respectively.

We proceed by induction on $m$. For $m=1$, we have~\cite[Id.\ 4]{benjaminquinn2003}
$$\forall n\in\No,\;T^{[1]}_{n}(-1,-1)=\displaystyle\sum_{k=0}^{\left\lfloor n/2\right\rfloor}\binom{n-k}{k}=F_{n+1},$$
which is in line with the above formulation with $Q^{[1]}=R^{[1]}=0$.

Similarly, we already have the result for $m=2$ with $Q^{[2]}=1$ and $R^{[2]}=0$ from Theorem~\ref{resT2}:
$$\forall n\in\No,\;T^{[2]}_{n}(-1,-1)=F_{n+3}-2^{\lfloor (n+1)/2\rfloor}.$$

Assume that the result is true for a given $m\geqslant2$, i.e., there exist two polynomials  $Q^{[m]}$ and $R^{[m]}$ with coefficients in $\mathbb Q$ and respectively of degree $(m-2)^+$ and $(m-3)^+$, such that $\forall n\in\No,\;T^{[m]}_{n}(-1,-1)=F_{n+2m-1}-2^p\left(Q^{[m]}\left(p\right)+(-1)^nR^{[m]}\left(p\right)\right)$, where $p=\left\lfloor\frac{n+1}{2}\right\rfloor$. We will now prove the result for $m+1$.

First, from Relations~(\ref{relrecTmodd}) and~(\ref{relrecTmeven}) at order $m+1$, we have $$\forall s\geqslant1,\;T^{[m+1]}_{2s+1}(-1,-1)=2T^{[m+1]}_{2s-1}(-1,-1)+T^{[m]}_{2s}(-1,-1).$$
Using Lemma~\ref{lemma} applied to $u_s=T^{[m+1]}_{2s+1}(-1,-1)$ (i.e., $v_s=T^{[m]}_{2s}(-1,-1)$), we obtain $$\forall s\geqslant1,\;T^{[m+1]}_{2s+1}(-1,-1)=2^s\underbrace{T^{[m+1]}_{1}(-1,-1)}_1+\displaystyle\sum_{k=1}^s2^{s-k}T^{[m]}_{2k}(-1,-1).$$
Since $T^{[m]}_1(-1,-1)=1$, we may write that $$\forall s\geqslant1,\;T^{[m+1]}_{2s-1}(-1,-1)=2^{s-1}+\displaystyle\sum_{k=1}^{s-1}2^{s-1-k}T^{[m]}_{2k}(-1,-1).$$

Recall that $\forall k\in\No$, $T^{[m]}_{2k}(-1,-1)=F_{2k+2m-1}-2^k\left(Q^{[m]}(k)+R^{[m]}(k)\right)$. Lemma~\ref{lemma} applied to $u_s=F_{2s+2m+2}$ (i.e., $v_s=F_{2s+2m-1}$) yields $\displaystyle\sum_{k=1}^{s}2^{s-k}F_{2k+2m-1}=F_{2s+2m+2}-2^sF_{2m+2}$, hence, $$\sum_{k=1}^{s-1}2^{s-1-k}F_{2k+2m-1}=\frac12\left(F_{2s+2m+2}-2^sF_{2m+2}-F_{2s+2m-1}\right)=F_{2s+2m}-2^{s-1}F_{2m+2}.$$

Thus, we have
\begin{eqnarray}
\forall s\geqslant1,\;T^{[m+1]}_{2s-1}(-1,-1)&=&2^{s-1}+\displaystyle\sum_{k=1}^{s-1}2^{s-1-k}\left(F_{2k+2m-1}-2^k\left(Q^{[m]}(k)+R^{[m]}(k)\right)\right)\nonumber\\
&=&2^{s-1}+F_{2s+2m}-2^{s-1}F_{2m+2}-2^{s-1}\displaystyle\sum_{k=1}^{s-1}\left(Q^{[m]}(k)+R^{[m]}(k)\right),\nonumber\\
T^{[m+1]}_{2s-1}(-1,-1)&=&F_{(2s-1)+2(m+1)-1}-2^{s-1}\left(F_{2m+2}-1+A(s)\right),
\label{relTmodd}
\end{eqnarray}
where $A(X)=\displaystyle\sum_{k=1}^{X-1}\left(Q^{[m]}(k)+R^{[m]}(k)\right)$. For any $\alpha\in\No$, the polynomial $\displaystyle\sum_{k=1}^{X-1}k^\alpha$ is of degree $\alpha+1$ in $X$. Therefore, since $Q^{[m]}+R^{[m]}$ is of degree $(m-2)^+$ and $m\geqslant2$, $A$ is of degree $(m-1)^+$.

Relation~(\ref{relTmodd}) gives the expression of $T^{[m+1]}_{n}(-1,-1)$ for all odd $n$. We may thus retrieve the even subsequence from Relation~(\ref{relrecTmeven}) to obtain, for $s\geqslant1$,
\begin{eqnarray}
T^{[m+1]}_{2s}(-1,-1)&=&T^{[m+1]}_{2s-1}(-1,-1)+T^{[m]}_{2s}(-1,-1)\nonumber\\
&=&F_{(2s-1)+2(m+1)-1}-2^{s-1}\left(F_{2m+2}-1+A(s)\right)\nonumber\\&&+F_{2s+2m-1}-2^s\left(Q^{[m]}(s)+R^{[m]}(s)\right),\nonumber\\
T^{[m+1]}_{2s}(-1,-1)&=&F_{2s+2(m+1)-1}-2^{s-1}\left(F_{2m+2}-1+A(s)+2Q^{[m]}(s)+2R^{[m]}(s)\right).
\label{relTmeven}
\end{eqnarray}

We introduce the polynomials $Q^{[m+1]},R^{[m+1]}$ as follows: $$Q^{[m+1]}=\frac12\left(A+F_{2m+2}-1+Q^{[m]}+R^{[m]}\right)\text{ and }R^{[m+1]}=\frac12\left(Q^{[m]}+R^{[m]}\right).$$
It is clear that $Q^{[m+1]},R^{[m+1]}\in\mathbb Q[X]$. In addition, $Q^{[m+1]}$ is of the same degree as $A$, i.e., $(m-1)^+$ and $R^{[m+1]}$ is of the same degree as $Q^{[m]}$, i.e., $(m-2)^+$.

Therefore, we can synthesize Relations~(\ref{relTmodd}) and~(\ref{relTmeven}) as follows:
$$T^{[m+1]}_n=F_{n+2(m+1)-1}-2^p\left(Q^{[m+1]}(p)+(-1)^nR^{[m+1]}(p)\right),$$
where $p=\left\lfloor\frac{n+1}{2}\right\rfloor$, $Q^{[m+1]}$ and $R^{[m+1]}$ are respectively $((m+1)-2)^+$th and $((m+1)-3)^+$th degree polynomials with coefficients in $\mathbb Q$. This concludes the proof by induction.
\end{proof}

\subsection{Towards additional formulae}

The present work has reviewed Bernoulli's second and third-order triangles and found multiple connections to the Fibonacci sequence, expressed in Theorems~\ref{theoremS2} and~\ref{resT2} to~\ref{resT3}. Theorem~\ref{theoremTm} generalizes Theorems~\ref{resT2} to~\ref{resT3} to higher orders. Furthermore, Theorem~\ref{theoremlambda} uncovers a relation between partial sums of binomial coefficients in Bernoulli's triangle and sequences satisfying an additive recurrence relation that generalizes the Fibonacci sequence.

The methodology followed in this article (i.e., the study of paths over Bernoulli's triangles) reveals sequences that are formed by the partial sums of binomial coefficients. More generally, these paths could yield further fruit by mapping various functions $f$ over the elements of Bernoulli's triangles in order to study the corresponding sequences $\displaystyle\sum_{k}^{}f(B^{[m]}_{n_0+kl,k_0-kc})$.

\section{Acknowledgments}
We thank Olivier Bordell\`{e}s and Harry Robertson for careful proofreading. We are grateful to the anonymous referee for numerous suggestions. Denis Neiter would also like to thank Na Wang and Richard Andr\'{e}-Jeannin for their support during the writing of this article.

\bigskip
\hrule
\bigskip

\noindent 2010 {\it Mathematics Subject Classification}:
Primary 11B39; Secondary 05A19.

\noindent \textit{Keywords: }
Fibonacci number, binomial coefficient, Pascal's triangle, Bernoulli's triangle.

\bigskip
\hrule
\bigskip

\noindent (Concerned with OEIS sequences
\seqnum{A000045},
\seqnum{A000930},
\seqnum{A003269},
\seqnum{A003520},
\seqnum{A005251},
\seqnum{A005314},
\seqnum{A008949},
\seqnum{A027934},
\seqnum{A099568},
\seqnum{A138653}, and
\seqnum{A193605}.)

\end{document}